\RequirePackage{fix-cm}
\documentclass{article}       
\usepackage{ulem}
\usepackage{xcolor}
\usepackage{graphicx}
\usepackage{amsmath}
\usepackage{amsthm}
\usepackage{graphicx,amssymb}

 \usepackage{mathptmx}
 \usepackage[ruled,vlined]{algorithm2e}
 \newcommand{\qbinom}{\genfrac{[}{]}{0pt}{}}
 
\usepackage{geometry}
\usepackage{authblk}
\usepackage{hyperref}
\newtheorem{theorem}{Theorem}
\newtheorem{proposition}{Proposition}
\newtheorem{lemma}{Lemma}
\newtheorem{corollary}{Corollary}

\newcommand\blfootnote[1]{%
  \begingroup
  \renewcommand\thefootnote{}\footnote{#1}%
  \addtocounter{footnote}{-1}%
  \endgroup
}

\begin{document}

\title{Total positivity and accurate computations related to $q$-Abel polynomials
\protect\thanks{This work was partially supported by Spanish research grants PID2022-138569NB-I00  (MCI/AEI)  and RED2022-134176-T (MCI/AEI) and by Gobierno de Arag\'{o}n (E41$\_$23R, S60$\_$23R).}

\blfootnote{\textit{This version of the article has been
accepted for publication, after peer review but is not the Version of Record and does not reflect post-acceptance improvements, or any corrections. The Version of Record is available online at: \href{http://dx.doi.org/10.1007/s10915-024-02699-8}{http://dx.doi.org/10.1007/s10915-024-02699-8}.
}}}



\author[1]{Y. Khiar}
\author[1]{E. Mainar}
\author[2,*]{E. Royo-Amondarain}
\author[1]{B. Rubio}

\affil[1]{Departamento de Matem\'{a}tica Aplicada/IUMA, Universidad de Zaragoza, Spain}
\affil[2]{Departamento de Matem\'{a}ticas / Centro de Astropartículas y Física de Altas Energías (CAPA), Universidad de Zaragoza, Spain}
\affil[*]{Corresponding author: eduroyo@unizar.es}



\date{}

\maketitle

\begin{abstract} 
The attainment  of accurate numerical solutions of ill-conditioned linear algebraic problems involving totally positive matrices has been gathering considerable attention among researchers over the last years. In parallel, the interest of $q$-calculus has been steadily growing in the literature. In this work  the $q$-analogue of the Abel polynomial basis is studied. The total positivity of  the matrix of change of basis between monomial and $q$-Abel bases is characterized, providing its bidiagonal factorization. Moreover, well-known high relative accuracy results of Vandermonde matrices corresponding to increasing positive nodes are extended to the decreasing negative case. This further allows to solve with high relative accuracy several algebraic problems concerning collocation, Wronskian and Gramian matrices of $q$-Abel polynomials. Finally, a series of numerical tests support the presented theoretical results and illustrate the goodness of the method where standard approaches fail to deliver accurate solutions.


\end{abstract}

\section{Introduction}
A sequence  $(p_0,p_1,\ldots)$ of polynomials such that $\deg(p_i)=i$, $i=0,1,\ldots$   is said to be of binomial type if it satisfies the following identities
 $$
 p_{n}(x+y)=\sum _{{k=0}}^{n}{n \choose k}\,p_{k}(x)\,p_{{n-k}}(y),\quad n\in\mathbb N.
 $$
 We can find a first example of this type of polynomials when considering the 
binomial   theorem 
  $(x+y)^n=\sum _{{k=0}}^{n}{n \choose k}\, x^k y^{n-k}$,  that shows that the sequence of monomials $(m_i)$ with $m_i(x):=x^i$ is of binomial type. 
 
 There are several generalizations of the binomial theorem. For the sequence $(x)_{n,h}$ of falling or rising factorials  defined by
$$
  (x)_{0,h}:=1,\quad (x)_{n,h}:=x(x-h)\cdots(x-(n-1)h),\\ 
$$
we also have 
 $$
 (x+y)_{n,h}=\sum _{{k=0}}^{n}{n \choose k}\,  (x)_{k,h} \,  (y)_{n-k,h},\quad n\in\mathbb N, 
 $$
which implies that the sequence  $(x)_{n,h}$ is also of binomial type. Moreover, it is well-known that the Touchard polynomials (also called the exponential polynomials): 
\begin{equation}\label{eq:Touchard}
   T_{0}(x):=1,\quad  T_{n}(x):=\sum _{k=0}^{n}S(n,k)x^{k},\quad n\in\mathbb N,
\end{equation}
where $S(n,k)$ are Stirling numbers of the second kind,  constitute the only polynomial sequence of binomial type with the coefficient of $x$ equal to 1 in every polynomial.

Using umbral calculus (cf. \cite{RST,RT}),    every polynomial sequence of binomial type may be obtained from  the sequence of Abel polynomials $(A_{n}^\alpha)$,  defined as
 \begin{equation}\label{eq:Abel}
   A_0^\alpha (x):=1,\quad   A_n^\alpha (x):=x(x-\alpha n)^{n-1},\quad  \alpha\in \mathbb R, \quad n\in\mathbb N \textcolor{red}{.}
\end{equation}
They are named after Niels Henrik Abel who proved the following equalities
$$
A_n^\alpha (x+\alpha)=\sum_{i=0}^n{n \choose i}A_i^\alpha (x)A_{n-i}^\alpha (\alpha),\quad  \alpha\in \mathbb R, \quad n\in\mathbb N.
$$

 During  last decades, $q$-calculus has been studied rigorously because of its implicit application in Mathematics, Mechanics and Physics \cite{Abe1997,Ernst2000Thesis,Jarad2013,Lavagno2000,Lavagno2002}. In fact, due to the intensive generalizations of the $q$-calculus and their interesting applications, $q$-analogue to  classical operators have emerged in recent years. Interested readers can find  in \cite{Ernst2013} a general introduction to the logarithmic $q$-analogue formulation of mathematical expressions with a special focus on its use for matrix calculations. In this sense, the $q$-analogue umbral method proposed in \cite{Ernst2006} is used to find $q$-analogues of significant matrices such as $q$-Cauchy matrices.
  
  There are also commutative $q$-binomial theorems, one of which is subsumed in a $q$-Abel binomial theorem for  the following $q$-analogue of the  Abel polynomial sequence  in \eqref{eq:Abel}:
\begin{equation}\label{eqBaseabel}
A_0^{q,\alpha}(x):=1,\quad A_n^{q,\alpha}(x):=x\prod_{j=1}^{n-1}(xq^j-\alpha [n]), \quad n\in\mathbb{N},
\end{equation} 
where $[n]$ is a $q$-integer.  Let us observe that  for $q=1$ the $q$-analogue of  Abel polynomials \eqref{eqBaseabel} coincide with the classical Abel polynomials  \eqref{eq:Abel}.
 
In the literature, the total positivity of some bases formed by polynomials of binomial type has been analyzed.  It is well-known that the monomial basis $(m_0,\ldots,m_n)$  of the space $P^n$ of polynomials of degree not greater than $n$ is strictly totally positive on $(0,\infty)$ (see Section 3 of \cite{Koev4}). On the other hand, the total positivity of Touchard polynomial bases $(T_0,\ldots,T_n)$ defined from \eqref{eq:Touchard}  has been recently proved in Section 6 of \cite{MPRNewton}, where  a procedure to get   computations to high relative accuracy  with their collocation and Wronskian matrices has also been provided.  

In a given floating point arithmetic, high relative accuracy implies that  the relative errors on the computations have the order of the machine precision and are not  drastically affected by  the dimension or  conditioning of the considered problem. 
So, the design of algorithms to high relative accuracy is an important trend in numerical linear algebra, which is attracting the interest of many researchers \cite{Koev1,Koev4,MPRRACSAM,MPRAdvances,MPRNA,MPRNewton,MPRNLA,MPRCalcolo,MPRJacobi,Marco3,MMP}.


In this paper the total positivity of polynomial bases  $(A_0^{q,\alpha},\ldots, A_n^{q,\alpha})$  formed by the $q$-analogue Abel polynomials in \eqref{eqBaseabel}  is going to be characterized in terms of the sign of $\alpha$. A bidiagonal factorization of the matrix of change of basis between monomial and $q$-analogue Abel polynomial bases will be derived. This factorization will be used to solve to high relative accuracy many algebraic problems with collocation, Wronskian and Gramian matrices of $q$-analogue Abel polynomial bases and so, for the well known Abel bases. 


 The layout of the paper is as follows. Section \ref{sec:notations} provides  notations and auxiliary results. Section \ref{sec:bda} focuses on the matrices of  change of basis between monomial and $q$-analogue Abel polynomial bases.    Their bidiagonal decomposition is provided and  their total positivity property is deduced. 
 In Section \ref{sec:tp}, well-known results on the strict total positivity of Vandermonde matrices  for increasing positive nodes are recovered for conveniently scaled Vandermonde matrices for  decreasing negative nodes. Then     
  $q$-analogue Abel  polynomials are considered  and the total positivity of  their collocation,  
   Wronskian and Gramian matrices  is analyzed.   Finally, Section \ref{sec:experimentos} presents numerical experiments confirming the accuracy of
the proposed methods for the computation of eigenvalues, singular values, inverses or the
solution of some linear systems related to collocation, Wronskian and Gram matrices of $q$-analogue Abel polynomial bases.

\section{Notations and auxiliary results}\label{sec:notations}
Let us recall that a matrix is \textit{totally positive}: TP (respectively, \textit{strictly totally positive}: STP) if all its minors are nonnegative (respectively, positive). Several 
 applications of these matrices can be found in \cite{Ando,Fallat,Pinkus}.

The \textit{Neville elimination} (NE) is an alternative procedure to Gaussian elimination  (see \cite{Gasca1,Gasca2,Gasca3}).
Given an $(n+1)\times (n+1)$, nonsingular matrix  $A=(a_{i,j})_{1\le i,j\le n+1}$, the NE process calculates  a sequence  of matrices
\begin{equation}\label{stepsNE}
A^{(1)} :=A\to   A^{(2)} \to \cdots \to A^{(n+1)},
\end{equation}
so that the entries of $A^{(k+1)}$ below the  main diagonal in the $k$ first columns,  $1\le k\le  n$, are zeros and so,     $A^{(n+1)} $  is   upper triangular.  In each step of the NE procedure, the matrix $A^{(k+1)}  =(a_{i,j}^{(k+1)})_{1\le i,j\le n+1}$   is computed from the matrix $A^{(k)}=(a_{i,j}^{(k)})_{1\le i,j\le n+1}$ as follows:
\begin{equation}\label{elementsNE}
 a_{i,j}^{(k+1)}:=\begin{cases}
 a^{(k)}_{i,j}, \quad  &\textrm{if} \ 1\le i \le k, \\
 a^{(k)}_{i,j}-  \frac{a_{i,k}^{(k)}}{ a_{i-1,k}^{(k)} } a^{(k)}_{i-1,j} , \quad &\textrm{if } k+1\le i,j\le n+1  \textrm{ and  }  a_{i-1,k}^{(k)}\ne 0,\\
 a^{(k)}_{i,j},  \quad &\textrm{if} \quad  k+1\le i\le n+1 \textrm{ and  }  a_{i-1,k}^{(k)}= 0.  
\end{cases}
\end{equation}

 The $(i,j)$ \textit{pivot}  of the NE process of the matrix $A$ is
the element $a_{i,j}^{(j)}$, for $1\le j\le i\le n+1$,
  and we say that $p_{i,i}$ is the $i$-th \textit{diagonal pivot}. Let us observe that 
whenever all  pivots  are nonzero, no  row exchanges are needed in the NE procedure. The $(i,j)$ \textit{multiplier} of the NE  process of the matrix $A$ is
\begin{equation}\label{multiEN} 
m_{i,j}:=\begin{cases}
 { a_{i,j}^{(j)} }/{ a_{i-1,j}^{(j)}},  & \textrm{if }\,  a_{i-1,j}^{(j)} \ne 0,\\
\strut 0, & \textrm{if }\,  a_{i-1,j}^{(j)} = 0,
\end{cases} ,\quad 1\le j < i\le n+1.
\end{equation}

Neville elimination is a nice tool to deduce that a given matrix is  STP, as shown in this  characterization  derived from Theorem 4.1, Corollary 5.5 of \cite{Gasca1} and the arguments of p. 116 of \cite{Gasca3}.

\begin{theorem} \label{thm1}
A given nonsingular matrix A is STP (resp., TP) if and only if the NE of $A$ and $A^T$ can be performed without row
exchanges, all the multipliers of the NE  of $A$ and $A^T$ are positive (resp., nonnegative), and  the diagonal pivots of the Neville
elimination of $A$ are all positive.
\end{theorem} 

By Theorem 4.2 and the arguments of p.116  of \cite{Gasca3},   a nonsingular TP matrix $A=(a_{i,j})_{1\le i,j\le n+1}$
 admits a factorization of the form
\begin{equation}\label{BDAfac}
A=F_nF_{n-1}\cdots F_1 D G_1G_2 \cdots  G_n,
\end{equation}
 where $F_i$ and $G_i$, $i=1,\ldots,n$, are the TP, lower and upper triangular bidiagonal matrices given by 
\begin{equation}  
\label{fgi}
\footnotesize
F_i=
\begin{pmatrix}
1  \\
0 & 1  \\
  & \ddots & \ddots   \\
  & & 0 &  1   \\
  & & & m_{i+1,1} &  1  \\
  & & & & m_{i+2,2} &  1    \\
  & & & & & \ddots & \ddots     \\
  & & & & & & m_{n+1,n+1-i}  &  1
\end{pmatrix}, \quad
G_i^T=
\begin{pmatrix}
1  \\
0 & 1  \\
  & \ddots & \ddots   \\
  & & 0 &  1    \\
  & & & \widetilde{m}_{i+1,1} &  1   \\
  & & & & \widetilde{m}_{i+2,2} &  1   \\
  & & & & & \ddots & \ddots      \\
  & & & & & & \widetilde{m}_{n+1,n+1-i} & 1
\end{pmatrix},
\end{equation}
and $D={\rm diag}\left (p_{1,1}, p_{2,2},\ldots, p_{n+1,n+1}\right)$ has positive diagonal entries. The diagonal entries $p_{i,i}$ of $D$ are the positive diagonal pivots of the NE of $A$ and the nonnegative elements $m_{i,j}$ and $\widetilde{m}_{i,j}$ are  the multipliers of the NE of  $A$ and $A^T$, respectively.  
 If, in addition, the entries  $m_{ij}$, $\widetilde{m}_{ij}$ satisfy
\begin{equation}\label{eq:CondU}
m_{ij}=0 \quad \Rightarrow \quad  m_{hj}=0, \;\forall \, h>i  \quad \textrm{and}\quad 
  \widetilde{m}_{ij}=0 \quad \Rightarrow  \quad  \widetilde{m}_{ik}=0, \; \forall \, k>j,
\end{equation}
 then the decomposition (\ref{BDAfac}) is unique.

In \cite{Koev4}, the bidiagonal factorization \eqref{BDAfac} of an $(n+1)\times (n+1)$ nonsingular and TP matrix $A$ is represented by defining a matrix $BD(A)=(  BD(A)_{i,j})_{1\le i,j\le n+1} $ such that 
\begin{equation}\label{eq:BDA}
 BD(A)_{i,j}:=\begin{cases} m_{i,j}, & \textrm{if } i>j, \\  p_{i,i}, & \textrm{if } i=j, \\  \widetilde{m}_{ j ,i}, & \textrm{if } i<j. \\     \end{cases}
\end{equation}

The transpose of $A$  is also totally positive and, using the factorization \eqref{BDAfac}, can be written as follows   
\begin{equation}\label{eq:BDAtranspose}
A^{T}={G}_n^{T}{G}_{n-1}^{T}\cdots {G}_1^{T} D {F}_1^{T}  {F}_2^{T}\cdots {F}_n^T.
\end{equation}

Let us also recall that  a real  value $t$ is computed with high relative accuracy (HRA) whenever the computed value $\tilde t$  satisfies
$$ 
\frac{ | t-\tilde t |   }{  | t| } < Ku,
$$
where $u$ is the unit round-off and $K>0$ is a  constant, which is independent of the arithmetic precision. Clearly, HRA implies a great accuracy since  the relative errors in the computations have the same order as the machine precision. A  sufficient condition to assure that an algorithm can be computed to HRA  is the no inaccurate cancellation  condition, sometimes denoted as NIC condition, 
which is satisfied if  the algorithm only evaluates products, quotients,   sums of numbers of the same sign, subtractions of numbers of opposite sign or subtraction of initial data    (cf. \cite{Koev1,Koev4}). 
 
If  the bidiagonal factorization \eqref{BDAfac} of a nonsingular and TP matrix $A$ can be computed to HRA, the computation of   its eigenvalues and singular values, the computation of   $A^{-1}$ and even the resolution of systems of linear equations  $Ax=b$, for vectors $b$ with alternating signs, can be also computed to HRA using the algorithms provided in \cite{Koev3}. 
 
Let  $(u_0, \ldots, u_n)$ be  a basis  of a space $U$ of  functions defined on $I\subseteq R$.
Given  a sequence of parameters
  $ t_1<\cdots<t_{n+1} $  on $I$,  the corresponding \textit{collocation matrix}  is defined by
\begin{equation} \label{eq:MC}
M(t_1,\ldots , t_{n+1}):=\big(  u_{j-1}(t_i) \big)_{1\le i,j\le n+1}.
\end{equation}
Many fundamental problems in interpolation and approximation require linear algebra computations related to collocation  matrices. In fact, they appear when imposing interpolation conditions for a given basis.

Related to the Hermite interpolation problem, if the space $U$ is formed by   $n$-times continuously differentiable functions  and $t\in I$, the \textit{Wronskian matrix} at  $t$  is defined by:
\begin{equation} \label{eq:MW}
W(u_0, \ldots, u_n)(t) :=(u_{j-1}^{(i-1)}(t))_{1\le i,j\le n+1},
\end{equation}  
where $u^{(i)}(t)$, $i\le n$,  denotes the $i$-th derivative  of $u$ at $t$. As usual, we shall use $u'(t)$ to denote  the first derivative of $u$ at $t$.

Now, let us suppose that $U$ is a Hilbert space of functions  on the interval $[0,T]$, $T\le +\infty$, under a given inner product
$$
\langle u,v\rangle := \int_{0}^T u(t)v(t)\, dt,
$$
defined for any  $u,v\in U$. Then, given linearly independent  functions $v_0,\ldots, v_n$  in $U$, the corresponding \textit{Gram or Gramian matrix}   is the symmetric matrix  described by:
$$
G(v_0,\ldots,v_n):=\left( \langle v_{i-1},v_{j-1}\rangle   \right)_{1\le i,j\le n+1}.
$$

\section{Factorization of matrix conversion between $q$-Abel and monomial bases}\label{sec:bda}
Let  ${\mathbf P}^n(I)$ be the  (n+1)-dimensional vector space formed by all polynomials  in the variable $x$ defined on a real interval $I$ and whose degree is not greater than $n$, that is,
\[
 {\mathbf P}^n(I):=\textrm{span}\{1, x, \ldots, x^n\},\quad   x\in I.
 \]
Then, let us recall definition \eqref{eqBaseabel}, for $m \in \mathbb{N}\cup\{0\}$ and $q>0$ any positive real number, the $m$-th degree $q$-Abel polynomial (see Section 3 of \cite{Johnson}) is given by 
\begin{equation*}
A_0^{(q,\alpha)}(x):=1,\quad A_m^{(q,\alpha)}(x):=x\prod_{j=1}^{m-1}(xq^j-\alpha [m]), \quad \alpha\in \mathbb R,\quad m\in\mathbb{N},
\end{equation*}
where the $q$-integer $[m]$ is defined as
\begin{equation}\label{eq:prop1}
[m]:=\sum_{i=0}^{m-1}q^{i}=
    \begin{cases}
    \frac{1-q^m}{1-q}, & q\neq 1,\\
    m,&q=1.
\end{cases}
\end{equation}
Observe that, for the particular case $q=1$, the polynomials given by \eqref{eqBaseabel} are the well-known Abel polynomials in $\eqref{eq:Abel}$ (cf. \cite{RST,RT}).

Note that, in analogy to the integer case, it is possible to define the $q$-factorial as
\begin{equation*}
[m]! :=[m][m-1]\cdots[1],
\end{equation*}  and, in the same way, to introduce a $q$-binomial coefficient, given by
\begin{equation*}
\qbinom{m}{k} := \frac{[m]!}{[k]![m-k]!}.
\end{equation*}
In addition, let us introduce some useful properties that are verified by $q$-integers. First, the subtraction of two $q$-integers satisfies
\begin{equation}\label{eq:prop2}
[t]-[s]=q^s[t-s].
\end{equation}
Moreover, $q$-binomial coefficients, which are extensions of the integer case, maintain properties such as
\begin{equation}\label{eq:prop4}
[k]{m \brack k}=[m]{m-1 \brack k-1},
\end{equation}
and the following equivalents to Pascal's identity,
\begin{equation}\label{eq:prop3}
{m \brack k}={m-1 \brack k-1}+q^k{m-1 \brack k},\quad {m \brack k}=q^{m-k}{m-1 \brack k-1}+{m-1 \brack k}.
\end{equation}

Using formula $(1.8)$ of \cite{Johnson}, $q$-Abel polynomials \eqref{eqBaseabel} can also be written in terms of the monomial basis as follows,
\begin{align*}
    A_m^{(q,\alpha)}(x)&=x\prod_{j=1}^{m-1}(xq^j-\alpha [m])= q^{m-1}x\prod_{j=0}^{m-2}(xq^j-\alpha [m]) \\
    &= q^{m-1}\sum_{k=0}^{m-1}{ m-1\brack k}q^{\binom{k}{2}}x^{k+1}\left(-\frac{[m]\alpha}{q}\right)^{m-1-k} \\
    &=\sum_{k=0}^{m-1}{ m-1\brack k}q^{\binom{k+1}{2}}x^{k+1}\left(-[m]\alpha\right)^{m-1-k}\\\ &= \sum_{k=1}^{m}(-1)^{m-k}{ m-1\brack m-k}q^{\binom{k}{2}}[m]^{m-k}\alpha^{m-k}x^{k},\quad m\in\mathbb N.
\end{align*}
Then, we can write
\begin{equation}\label{eq:L}
 (A_0^{(q,\alpha)}(x),   \ldots, A_n^{(q,\alpha)}(x))^T=  L^{(q,\alpha)} (1,  \ldots,x^{n})^T,
\end{equation}
where the matrix of change of basis is a nonsingular triangular matrix $L^{(q,\alpha)}=( {l}_{i,j})_{1\le i,j\le n+1}$ with
\begin{equation}\label{eq:qlij}
l_{1,1}=1, \quad l_{i,1}=0, \quad i\geq2, \quad l_{i,j}=(-\alpha[i-1])^{i-j}q^{\binom{j-1}{2}}{i-2 \brack i-j},\quad 2\leq j\leq i\leq n+1
\end{equation}
and the binomial coefficient $\binom{n}{k}$ vanishes for any $n<k$. Then, $L^{(q,\alpha)}$ can be written as
\[
\setlength\arraycolsep{6pt}
L=
\begin{pmatrix}
1 & 0 & 0 & 0 &\cdots & 0 \\
0 & 1 & 0 &0 & \cdots & 0 \\
0 & -\alpha[2] & q &0 & \cdots & 0 \\
0 & (-\alpha[3])^{2} & -\alpha q[3][2] & q^3 &\cdots & 0 \\
\vdots &  \vdots  &\vdots   & \vdots &\ddots & 0  \\   
0 & (-\alpha[n])^{n-1} & (-\alpha[n])^{n-2}q[n-1] & (-\alpha[n])^{n-3}q^3\qbinom{n-1}{n-3} & \cdots & q^{\binom{n}{2}} \\
\end{pmatrix}.
\]

At this point, our objective is to obtain the expression of the pivots and multipliers of the NE of  the matrix $L^{(q,\alpha)}$ and, consequently, to deduce its  bidiagonal factorization~\eqref{BDAfac}, analyzing its total positivity property and the computation to HRA of its bidiagonal decomposition. To this aim, it is convenient to introduce an auxiliary result involving the $q$-binomial coefficients.
\begin{lemma} \label{lemma:Aq}
Let $i,k\in\mathbb{N}$ such that $i\geq k$. Then,
   \begin{equation}\label{eq:auxAq}
S_{k,i}:=  \sum_{l=0}^{k} (-1)^l q^{\binom{l+1}{2}} \qbinom{k}{l} \qbinom{i-l}{k} = 1.
   \end{equation}
\end{lemma}  

\begin{proof} 
First, it is simple to verify that for $k=0$,
\begin{equation*}
S_{0,i}=\qbinom{i}{0}=1,  \quad i\geq k.
\end{equation*}
Now, by induction on $k$, we assume that $S_{k,i}=1$ and write $S_{k+1,i}$ as
\begin{equation*}
   S_{k+1,i}=\sum_{l=0}^{k+1} (-1)^l q^{\binom{l+1}{2}} \qbinom{k+1}{l} \qbinom{i-l}{k+1}, \quad  i\geq k+1.
\end{equation*}
Using the analogue of Pascal's identity \eqref{eq:prop3} on the first $q$-combinatorial number and shifting one of the factors in the second one, we have
\begin{equation*}
   S_{k+1,i}=\sum_{l=0}^{k+1} (-1)^l q^{\binom{l+1}{2}} \left( \qbinom{k}{l-1} + q^l \qbinom{k}{l}\right) \qbinom{i-l}{k}
   \frac{[i-l-k]}{[k+1]},
\end{equation*}
and, since $q^l[i-l-k]=[i-k]-[l]$, 
\begin{align*}
   S_{k+1,i}&=\frac{1}{[k+1]}\sum_{l=0}^{k+1} (-1)^l q^{\binom{l+1}{2}} \left( \qbinom{k}{l-1}[i-l-k] + \qbinom{k}{l}[i-k] - \qbinom{k}{l}[l] \right) \qbinom{i-l}{k} \\
     &=\frac{[i-k]}{[k+1]}+\frac{1}{[k+1]}\sum_{l=1}^{k+1} (-1)^l q^{\binom{l+1}{2}} \qbinom{k}{l-1} \qbinom{i-l}{k} \left( [i-l-k] - [k-l+1] \right),
\end{align*}
where in the last step the second sum is one by hypothesis and in the third one the bottom factor of the $q$-combinatorial number is shifted. Then, applying \eqref{eq:prop2} and shifting the index $l\to l-1$, it follows that
\begin{align*}
   S_{k+1,i}&=\frac{[i-k]}{[k+1]}+\frac{1}{[k+1]}\sum_{l=1}^{k+1} (-1)^l q^{\binom{l}{2}} \qbinom{k}{l-1} \qbinom{i-l}{k}q^{k+1}[i-2k+1] \nonumber\\
   &=\frac{[i-k]}{[k+1]}+\frac{[k+1]-[i-k]}{[k+1]}\sum_{l=0}^{k} (-1)^{l} q^{\binom{l+1}{2}} \qbinom{k}{l} \qbinom{i-l-1}{k}=1, \quad   i\geq k+1,
\end{align*}
where in the last step the sum is $1$ by hypothesis and so, \eqref{eq:auxAq} holds for $k+1$. 
\qed
\end{proof}

\begin{theorem} \label{thm:BDALf}
Let $L^{(q,\alpha)}\in {\mathbb R}^{ (n+1)\times (n+1)}$ be the matrix defined in \eqref{eq:qlij}. Then, 
\begin{equation}\label{eq:BDAbelq}
L^{(q,\alpha)} =  F_n \ldots F_{1}D, 
\end{equation}
where  ${F}_i$, $i=1,\ldots,n$, are the lower triangular bidiagonal matrices whose structure is described by \eqref{fgi} and their off-diagonal entries are 
\begin{equation}\label{eq:mij}
m_{i,1}:=0 \quad 2\le i\le n+1, \quad m_{i,j}:=-q^{j-2}\left(\frac{[i-1]}{[i-2]}\right)^{i-j}[i-j]\alpha, \quad 2\le j< i \le n+1, 
\end{equation}
and $D \in {\mathbb R}^{ (n+1)\times (n+1)}$ is the diagonal matrix $D={\rm diag}\left (p_{1,1}, p_{2,2},\ldots, p_{n+1,n+1}\right)$ with
\begin{equation}\label{eq:pii}
 p_{i,i}:=q^{\binom{i-1}{2}}, \quad 1\le i\le n+1. 
\end{equation}
\end{theorem}

\begin{proof} Let   $L^{(k)}=(l_{i,j}^{(k)})_{1\le i,j\le n+1}$ be the matrices obtained after $k$ steps of the NE process of $L^{(1)}:=L^{(q,\alpha)}$.
To begin with, by \eqref{multiEN}, the multipliers $m_{i,1}=0$ since $l_{i,1}=0$ for $2\le i\le n+1$ by definition \eqref{eq:qlij}.
Then, let us show by induction on $k$ that 
\begin{equation}\label{eq:lqkij}
l_{i,j}^{(k)}=(-\alpha)^{i-j}q^{\binom{j-1}{2}}[i-1]^{i-k+1}\sum_{l=0}^{k-2}(-1)^l q^{\binom{l}{2}}{k-2\brack l} {i-2-l\brack i-j-l}[i-1-l]^{k-j-1},
\end{equation}
with $2\leq k\leq j\leq i\leq n+1$.
For $k=2$, using the first step of NE \eqref{elementsNE} and \eqref{eq:qlij}, we have
$$
l_{i,j}^{(2)}=l_{i,j}=(-1)^{i-j}q^{\binom{j-1}{2}}[i-1]^{i-j}{i-2 \brack i-j}\alpha^{i-j},\quad 2\leq j<i\leq n+1.
$$
Assuming that claim \eqref{eq:lqkij} holds for $k$, let us obtain
\begin{equation}\label{eq:lijk_plus1}
l_{i,j}^{(k+1)}= l^{(k)}_{i,j}-  \frac{l_{i,k}^{(k)}}{ l_{i-1,k}^{(k)} } l^{(k)}_{i-1,j}.
\end{equation}
First, we compute the quotient in the previous expression. Applying the induction hypothesis \eqref{eq:lqkij}, we have:
\begin{equation}\label{eq:mik-qproof}
\frac{l_{i,k}^{(k)}}{ l_{i-1,k}^{(k)} }=
-\alpha\left(\frac{[i-1]}{[i-2]}\right)^{i-k}[i-k]q^{k-2}\frac{[i-1]P_{i,k}}{q^{k-2}[i-k]P_{i-1,k}},
\end{equation}
where $P_{i,k}$ is given by
\begin{equation*}
P_{i,k}:=\sum_{j=0}^{k-2}(-1)^j q^{\binom{j}{2}} \qbinom{k-2}{j} \qbinom{i-2-j}{k-2} \left[i-j-1\right]^{-1}.
\end{equation*}

We intend now to show that the quotient in \eqref{eq:mik-qproof} is $1$, for $2\leq k\leq j\leq i\leq n+1$. Starting by the numerator, using $[i-1]=q^j[i-1-j]+[j]$ it follows that
\begin{align}
[i-1]P_{i,k}&=\sum_{j=0}^{k-2}(-1)^j q^{\binom{j+1}{2}} \qbinom{k-2}{j} \qbinom{i-2-j}{k-2} \nonumber\\ &+ \sum_{j=0}^{k-2}(-1)^j q^{\binom{j}{2}} \qbinom{k-2}{j} \qbinom{i-2-j}{k-2}\frac{[j]}{[i-j-1]}.\label{eq:qdemnum}
\end{align}
In a similar fashion, $q^{k-1}[i-k]=[i-1]-[k-1]$, and the denominator can be expressed as
\begin{align}
&q^{k-2}[i-k]P_{i-1,k}= \nonumber\\ 
&=\sum_{j=0}^{k-2}(-1)^j q^{\binom{j}{2}-1}\left(\qbinom{k-2}{j}\frac{[i-1]}{[i-2-j]}-\qbinom{k-1}{j+1}\frac{[j+1]}{[i-2-j]} \right) \qbinom{i-3-j}{k-2} \nonumber\\
&=\sum_{j=0}^{k-2}(-1)^j q^{\binom{j}{2}-1}\left(\qbinom{k-2}{j}\frac{[i-1]-[j+1]}{[i-2-j]}-\qbinom{k-2}{j+1}\frac{q^{j+1}[j+1]}{[i-2-j]} \right) \qbinom{i-3-j}{k-2} \nonumber\\
&=\sum_{j=0}^{k-2}(-1)^j q^{\binom{j+1}{2}}\qbinom{k-2}{j}\qbinom{i-3-j}{k-2} + \sum_{j=0}^{k-2}(-1)^j q^{\binom{j}{2}} \qbinom{k-2}{j} \qbinom{i-2-j}{k-2} \frac{[j]}{[i-1-j]}, \label{eq:qdemden}
\end{align}
where in the second step \eqref{eq:prop3} has been used, followed in the last step by a shift of the summation index in the second summand. Finally, note that both numerator \eqref{eq:qdemnum} and denominator \eqref{eq:qdemden} share exactly its second sum, whereas the first one is equal by Lemma \ref{lemma:Aq}, as long as $i>k$. As a consequence,
\begin{equation}\label{eq:quotient_multipliers}
\frac{l_{i,k}^{(k)}}{ l_{i-1,k}^{(k)} }=-\alpha\left(\frac{[i-1]}{[i-2]}\right)^{i-k}[i-k]q^{k-2}.
\end{equation}

Now, we proceed with the computation of $l_{i,j}^{(k+1)}$ in \eqref{eq:lijk_plus1}.
Assuming that \eqref{eq:lqkij} holds for $k$ we have
\begin{align*}
l_{i,j}^{(k+1)}= (-\alpha)^{i-j}q^{\binom{j-1}{2}}[i-1]^{i-k}
\left(
[i-1]\sum_{l=0}^{k-2}(-1)^l q^{\binom{l}{2}}{k-2\brack l} {i-2-l\brack j-2}[i-1-l]^{k-j-1}  \right. \\
\left. + q^{k-2}[i-k]\sum_{l=1}^{k-1}(-1)^l q^{\binom{l-1}{2}}{k-2\brack l-1} {i-2-l\brack j-2}[i-1-l]^{k-j-1}
\right),
\end{align*}
where in the second sum inside the parentheses we have shifted the index $l\to l+1$. Then, extending both sums from $0$ to $k-1$ and using $q^{k-1}[i-k]=[i-1]-[k-1]$, we arrive to
\begin{align*}
l_{i,j}^{(k+1)}&= (-\alpha)^{i-j}q^{\binom{j-1}{2}}[i-1]^{i-k} \times \\  &\left(
[i-1]\sum_{l=0}^{k-1}(-1)^l q^{\binom{l-1}{2}-1}\left(q^l{k-2\brack l}+{k-2\brack l-1} \right){i-2-l\brack j-2}[i-1-l]^{k-j-1}  \right. \\
&\left. - [k-1]\sum_{l=0}^{k-1}(-1)^l q^{\binom{l-1}{2}-1}{k-2\brack l-1} {i-2-l\brack j-2}[i-1-l]^{k-j-1}
\right) \\
&=(-\alpha)^{i-j}q^{\binom{j-1}{2}}[i-1]^{i-k}  \times \\ 
&\left(
[i-1]\sum_{l=0}^{k-1}(-1)^l q^{\binom{l-1}{2}-1}{k-1\brack l}{i-2-l\brack j-2}[i-1-l]^{k-j-1} \right. \\
&\left. - \sum_{l=0}^{k-1}(-1)^l q^{\binom{l-1}{2}-1}{k-1\brack l} {i-2-l\brack j-2}[l][i-1-l]^{k-j-1}
\right),
\end{align*}
where properties \eqref{eq:prop4} and \eqref{eq:prop3} have been used in the second step. Then, regrouping terms and using $[i-1]-[l]=q^l[i-1-l]$, the result is
\begin{align*}
l_{i,j}^{(k+1)}&=(-\alpha)^{i-j}q^{\binom{j-1}{2}}[i-1]^{i-k}
\sum_{l=0}^{k-1}(-1)^l q^{\binom{l}{2}}{k-1\brack l}{i-2-l\brack j-2}[i-1-l]^{k-j},
\end{align*}
and, as a consequence, \eqref{eq:lqkij} is verified for $k+1$.

Finally, by \eqref{eq:qlij}, the first pivot $p_{1,1}=l_{1,1}=1$ and with the expression of $l_{i,j}^{(k)}$ we can compute the rest of the pivots $p_{i,i}$ as
\begin{equation*}
l_{i,i}^{(i)}=(-\alpha)^{i-i}q^{\binom{i-1}{2}}[i-1]
\sum_{l=0}^{i-2}(-1)^l q^{\binom{l}{2}}{i-2\brack l}{i-2-l\brack i-2}[i-1-l]^{-1}=q^{\binom{i-1}{2}},
\end{equation*}
with $2\leq i\le n+1$,
since only the first summand corresponding to $l=0$ is non-zero. With respect to the multipliers $m_{i,j}$, by \eqref{eq:quotient_multipliers} and \eqref{multiEN} they verify the expression given in \eqref{eq:mij}.
\qed
\end{proof}

The provided bidiagonal decomposition~\eqref{eq:BDAbelq} of $L^{(q,\alpha)}$  can be stored in a  compact form through $BD(L^{(q,\alpha)})$ as
\begin{equation}\label{eq:BDAbelq2}
(BD(L^{(q,\alpha)}))_{i,j}=\begin{cases}
q^{\binom{i-1}{2}},&i=j,\\
-\alpha q^{j-2}\left(\frac{[i-1]}{[i-2]}\right)^{i-j}[i-j] ,& 2\le j< i \le n+1,\\
0,& \text{otherwise}.
\end{cases}
\end{equation}
To illustrate the form of the bidiagonal factorization \eqref{BDAfac} of $L^{(q,\alpha)}$, let us write the particular case $n=3$,
{
\footnotesize
\begin{equation*}
\setlength\arraycolsep{4pt}
L^{(q,\alpha)} = 
\begin{pmatrix}
1 & 0 & 0 & 0 \\
0 & 1 & 0 & 0 \\
0  & 0 & 1 & 0  \\
0 & 0 & 0 & 1 
\end{pmatrix}
\begin{pmatrix}
1 & 0 & 0 & 0 \\
0 & 1 & 0 & 0 \\
0 & 0 & 1 & 0  \\
0 & 0 & -\alpha \left(\frac{[3]}{[2]}\right)^{2}[2] & 1 
\end{pmatrix}
\begin{pmatrix}
1 & 0 & 0 & 0 \\
0 & 1 & 0 & 0 \\
0  & -\alpha [2]  & 1 & 0  \\
0 & 0 & -\alpha q\frac{[3]}{[2]} & 1 
\end{pmatrix}
\begin{pmatrix}
1 & 0 & 0 & 0 \\
0 & 1 & 0 & 0 \\
0 & 0 & q & 0  \\
0 & 0 & 0 & q^{3} 
\end{pmatrix}.
\end{equation*}
}

The analysis of the sign of $\alpha$ in \eqref{eq:BDAbelq2} will allow us to characterize the total positivity property of $L^{(q,\alpha)}$. This fact is stated in the following result.

\begin{corollary}  \label{cor:L}
Let $L^{(q,\alpha)}\in {\mathbb R}^{ (n+1)\times (n+1)}$ be the matrix defined in \eqref{eq:qlij} and $J$ the diagonal matrix $J:=\textrm{diag}((-1)^{i-1}  )_{1\le i\le n+1}$.   
\begin{enumerate}
\item[a)]   The matrix $L^{(q,\alpha)}$ is  totally positive if and only if  $\alpha \leq 0$. 
\item[b)]  The matrix  $L_J^{(q,\alpha)}:=JL^{(q,\alpha)}J$ is  totally positive if and only if $\alpha \geq 0$. 
\end{enumerate}
Moreover, in both cases, the matrices $L^{(q,\alpha)}$,  $L_J^{(q,\alpha)}$ and their corresponding bidiagonal factorizations \eqref{BDAfac}  can be computed to HRA.
\end{corollary} 
\begin{proof}  Let  
$L^{(q,\alpha)}=  F_n \cdots F_{1}D$ be the bidiagonal factorization provided by Theorem \ref{thm:BDALf}.

\begin{enumerate}
\item[a)]  The entries $m_{i,j}$ in \eqref{eq:mij} are nonnegative and $p_{i,i}$ in \eqref{eq:pii} are positive if and only if $\alpha\leq 0$, and so, the diagonal matrix $D$ and the bidiagonal matrices $F_i$, $i=1,\ldots,n$, are nonsingular TP matrices. Then, $L^{(q,\alpha)}$ is a product of nonsingular TP matrices and, by Theorem 3.1 of \cite{Ando}, is  a nonsingular TP matrix.
 
\item[b)] By Theorem 2 presented in \cite{MPRNewton}, $L_J^{(q,\alpha)}$ is nonsingular totally positive if and only if the multipliers $m_{i,j}$ in \eqref{eq:mij} are non-positive and the pivots $p_{i,i}$ in \eqref{eq:pii} are positive---this happens if and only if $\alpha\geq 0$. 
\end{enumerate}
Furthermore, taking into account $BD(L^{(q,\alpha)})$ given in  \eqref{eq:BDAbelq2} and formula (11) in \cite{MPRNewton}, the bidiagonal decomposition \eqref{BDAfac} of $L_J^{(q,\alpha)}$  can be stored in a compact form through $BD(L_J^{(q,\alpha)})$  as
 \begin{equation}\label{eq:BDAbelq2J}
(BD(L_{J}^{(q,\alpha)}))_{i,j}=\begin{cases}
q^{\binom{i-1}{2}},&i=j,\\
\alpha q^{j-2}\left(\frac{[i-1]}{[i-2]}\right)^{i-j}[i-j] ,& 2\le j< i \le n+1,\\
0,& \text{otherwise}.
\end{cases}
\end{equation}
Note that, as in the case of $BD(L^{(q,\alpha)})$, there are not inaccurate subtractions and as a consequence \eqref{eq:BDAbelq2} and \eqref{eq:BDAbelq2J} can be computed to HRA.
\qed
\end{proof}

\section{Computations with $q$-Abel polynomials to HRA}\label{sec:tp}

Let us recall that the monomial basis  $(m_{0} ,\ldots,m_{n})$, with $m_i(x)=x^i$ for $i=0,\ldots,n$,  is a strictly totally positive basis of   ${\mathbf P}^n(0,+\infty)$ and so,  the  collocation matrix  

\begin{equation}\label{eq:V}
V :=\left( t_{i}^{j-1} \right)_{1\le i,j \le n+1},
\end{equation}
 is strictly totally positive for any increasing sequence of positive values  $0 <t_1< \cdots t_{n+1}$ (see Section 3 of \cite{Koev4}).  In fact, $V$ is the Vandermonde matrix at the considered nodes. Let us recall that Vandermonde matrices   have relevant  applications in Lagrange interpolation and numerical quadrature (see for example  \cite{Finck,Oruc}).  As for $BD(V ) $ we have

 \begin{equation}\label{eq:BDV}
 BD(V )_{i,j}:=\begin{cases}
    \prod_{k=1}^{j-1}  \frac{t_{i } - t_{i-k }  }{t_{i-1} -t_{i-k-1}  }  , & \textrm{if } i>j, \\[5pt] 
  \prod_{k=1}^{i-1} (t_{i } - t_{k }), & \textrm{if } i=j,  \\[5pt]  
   t_{i}  , & \textrm{if } i<j,  
 \end{cases}
\end{equation}
and it can be easily checked that the computation of  $BD(V )$ does not require inaccurate cancelations and can be performed to HRA. Using the representation \eqref{eq:BDV} of Vandermonde matrices at increasing sequences of nodes, algebraic problems with this matrices can be solved to HRA.

Now, let us observe that these well-known results on Vandermonde matrices for increasing positive nodes can be recovered from conveniently sign transformed Vandermonde matrices for decreasing negative nodes as stated in the following result.  

\begin{proposition}\label{th:Vdecreasing}
Let $J$ be the diagonal matrix $J:=\textrm{diag}((-1)^{i-1}  )_{1\le i\le n+1}$ and $V$ the Vandermonde matrix \eqref{eq:V}. 
Then, the matrix $V_{J} := VJ$ is strictly totally positive if and only if $0>t_1>\cdots>t_{n+1}$. Moreover, its bidiagonal decomposition \eqref{BDAfac}  can be computed to HRA.

Furthermore, for any negative decreasing sequence $0>t_1>\cdots>t_{n+1}$, the computation of  the singular values, the inverse matrix   $V^{-1}$, as well as  the  solution  of linear systems $Vc = d$, where the entries of $d = (d_1, \ldots , d_{n+1})^T$ have alternating signs, can be performed to HRA.
\end{proposition}
\begin{proof}

Let us consider the bidiagonal factorization $V=F_{n}\cdots F_{1} D G_{1}\cdots G_{n}$. Since $J=J^{-1}$, then
\begin{equation}\label{BDAfacJ} 
V_J=F_{n}\cdots F_{1} D G_{1}\cdots G_{n}J= F_{n}\cdots F_{1}  (DJ) (JG_{1}J)\cdots(JG_{n}J)
=F_{n}\cdots F_{1} \tilde D \tilde G_{1}\cdots \tilde G_n,
\end{equation}
with $\tilde D:=DJ$ and $ \tilde G_k:=JG_{k}J$, $k=1,\ldots,n$.

Taking into account  \eqref{eq:BDV}, the bidiagonal decomposition \eqref{BDAfacJ} of  $V_J$   can be stored in a matrix form through $BD(V_{J} )_{i,j}$  as
 \begin{equation}\label{eq:BDVJ}
 BD(V_J )_{i,j}:=\begin{cases}
    m_{i,j}  , & \textrm{if } i>j, \\
    p_{i,i}  , & \textrm{if } i=j,  \\ 
    \tilde{m}_{j,i} , & \textrm{if } i<j,  
 \end{cases}
\end{equation}
where
 \begin{equation}\label{eq:BDVJentries}
  m_{i,j} = \prod_{k=1}^{j-1}  \frac{t_{i } - t_{i-k }  }{t_{i-1} -t_{i-k-1}  }, \quad p_{i,i}=(-1)^{i-1}\prod_{k=1}^{i-1} (t_{i } - t_{k }), \quad \tilde{m}_{j,i}=-t_i.
\end{equation}

By \eqref{eq:BDVJentries},   for $0>t_1>\cdots>t_{n+1}$, we clearly have $m_{i,j}>0$, $\tilde{m}_{i,j}>0$ and ${p}_{i,i}>0$, which means that $V_J$ is a product of STP matrices and so, it is also STP. Conversely, if $V_{J}$ is strictly totally positive then the entries  $m_{i,j}, \tilde m_{i,j}$ and $p_{i,i}$ in \eqref{eq:BDVJentries}  take positive values and we can deduce the admissible order of the nodes $t_i$. Since $p_{2,2}=-(t_2-t_1)>0$, we derive $t_1>t_2$. Now, from the expression of the multiplier $m_{i,2}$ we can write
 $$
   t_i-t_{i-1}=m_{i,2}(t_{i-1}-t_{i-2}),  \quad i=3,\ldots,n+1,
 $$
and, by induction, $t_i-t_{i-1}{>}0$ for $i=2,\ldots,n+1$. Finally, all $t_i$ must be non-positive since $\tilde m_{i,j}=-t_i\geq0$.

On the other hand, the  subtractions in the computation of the entries $m_{i,j}$, $p_{i,i}$ and $\tilde{m}_{i,j}$ involve only initial data and, as a consequence,  the matrix representation $BD(V_{J})$  in \eqref{eq:BDVJ} can be computed to HRA.  

Finally,  the resolution of the  algebraic problems mentioned in the statement  to HRA can be achieved using the matrix representation \eqref{eq:BDVJ} and the Matlab subroutines available in Koev's web page \cite{Koev3}. In this regard,  since  $J$ is a unitary matrix, the singular values of  $V_{ J}$ coincide with those of  $V$. Similarly, taking into account that  
\begin{equation*}
V^{-1}= JV^{-1}_{J},
\end{equation*}
we can compute $V^{-1}$ accurately. Finally,  if we have a linear system of equations  $Vc= d$, where the elements of  $d = (d_1, \ldots , d_{n+1})^T$  have alternating signs,  we can solve to HRA the  system  $V_{J}b = d$ and then obtain $c=Jb$. 
\qed
\end{proof}

Taking into account Corollary  \ref{cor:L} and Proposition \ref{th:Vdecreasing}, we shall analyze the total positivity of $q$-Abel bases, as well as factorizations providing computations to HRA when considering their collocation matrices. 

\begin{theorem}  \label{thm:collocation}
Given a sequence of parameters $t_1,\ldots, t_{n+1}$, with $t_i\ne t_j$ for $i\ne j$. Let
\begin{equation}\label{eq:collA}
A^{(q,\alpha)}:=(A_{j-1}^{(q,\alpha)}(t_i))_{1\leq i,j\leq n+1},
\end{equation}
 be the collocation matrix  of the q-Abel basis \eqref{eqBaseabel} and  $J:=\textrm{diag}((-1)^{i-1}  )_{1\le i\le n+1}$. Then
\begin{enumerate}
\item[a)] If  $0<t_1<\cdots<t_{n+1}$ and $\alpha\leq0$,  $A^{(q,\alpha)}$ is STP.
\item[b)]  If  $0>t_1>\cdots>t_{n+1}$ and $\alpha\geq0$,   
$A^{(q,\alpha)}_{J} := A^{(q,\alpha)} J$ is STP.
\end{enumerate}
Furthermore, in both cases, the matrices  $A^{(q,\alpha)}$, $A^{(q,\alpha)}_{J}$ and their bidiagonal decompositions \eqref{eq:BDA} can be computed to HRA. Consequently, for  $0<t_1<\cdots<t_{n+1}$ with $\alpha\leq0$  and    $0>t_1>\cdots>t_{n+1}$ with $\alpha\geq0$, the eigenvalues (only for $0<t_1<\cdots<t_{n+1}$ with $\alpha\leq0$), the singular values and the inverse matrix of $A^{(q,\alpha)}$, as well as the solution of linear systems $A^{(q,\alpha)}c = d$, where the entries of $d = (d_1, \ldots , d_{n+1})^T$ have alternating signs, can be computed to HRA.
\end{theorem} 

\begin{proof}

Let $(m_0,\ldots,m_n)$ be the monomial basis of ${\mathbf P}^n$ and  $ {L^{(q,\alpha)}}$ be the change of basis matrix such that $(A_0^{(q,\alpha)},   \ldots, A_n^{(q,\alpha)})= (m_{0},\ldots, m_{n}) (L^{(q,\alpha)})^T$ (see \eqref{eq:L}). Then,  
\begin{equation}\label{eq:change}
A^{(q,\alpha)}=V (L^{(q,\alpha)})^T,
\end{equation}
where $V$ is the Vandermonde matrix defined in \eqref{eq:V}.
  
\begin{enumerate}
\item[a)] If  $0<t_1<\cdots<t_{n+1}$, the Vandermonde matrix $V$ is STP  and its decomposition \eqref{eq:BDV} can be computed to HRA. Moreover, for $\alpha\leq0$, by Corollary \ref{cor:L}, $L^{(q,\alpha)}$ is a nonsingular  TP matrix and can be computed to HRA. By formula \eqref{eq:BDAtranspose}, its transpose  is also a nonsingular TP matrix and can be computed to HRA. As a direct consequence of these facts and taking into account that, by Theorem 3.1 of  \cite{Ando}, the product of an STP matrix and a nonsingular TP matrix is an STP matrix, we deduce that $A^{(q,\alpha)}$ is STP. Furthermore, the bidiagonal decomposition \eqref{BDAfac} of $A^{(q,\alpha)}$ can be obtained to HRA using Algorithm 5.1 \cite{Koev2}, since it is the product of two nonsingular TP matrices whose decompositions are known to HRA.

\item[b)] If  $0>t_1>\cdots>t_{n+1}$,  using Theorem \ref{th:Vdecreasing},  we deduce that $VJ$ is STP  and its bidiagonal factorization \eqref{BDAfacJ}  can be obtained to HRA. Additionally, if $\alpha\geq0$, using Corollary \ref{cor:L}, we deduce that $JL^{(q,\alpha)}J$ is a nonsingular TP matrix  and its bidiagonal factorization \eqref{BDAfac} can be computed to HRA.  Since $J=J^{-1}$, from \eqref{eq:change}, we can write
\begin{equation}\label{eq:JWJ}
A^{(q,\alpha)}J  =(VJ) (J(L^{(q,\alpha)})^TJ),
\end{equation}
and deduce that $A^{(q,\alpha)}J$ is STP because it can be written as the product of STP  and nonsingular TP matrices. Analogously to the previous case, the bidiagonal decomposition of $A_J^{(q,\alpha)}$ can be computed to HRA, being the product of two bidiagonal factorizations that are also computed to HRA.
\end{enumerate}
With a similar reasoning as in Proposition \ref{th:Vdecreasing} we can guarantee the resolution of the mentioned algebraic problems with the matrices $A^{(q,\alpha)}$ and $A^{(q,\alpha)}_{J}$.
\qed
\end{proof}

As a direct consequence of the strict  total positivity property of the collocation matrix $A^{(q,\alpha)}$ \eqref{eq:collA} of the q-Abel basis for $\alpha \le 0$, the following results holds. 
\begin{corollary}
For a given $\alpha \le 0$, the $q$-Abel basis $(A_0^{(q,\alpha)},   \ldots, \,A_n^{(q,\alpha)})$ defined in \eqref{eqBaseabel} is an STP basis of \,${\mathbf P}^n(0,+\infty)$.
\end{corollary}

Corollary 1 of \cite{MPRCalcolo} provides the   factorization~\eqref{BDAfac} of ${W}:=W(m_{0},\ldots,m_{n})(x)$,
  the Wronskian matrix of the monomial basis $(m_{0},\ldots,m_{n})$ at $x\in\mathbb R$. For~the matrix representation $BD( {W})$, we have
\begin{equation}\label{eq:BDWmonomials}
 BD( {W})_{i,j}:=\begin{cases}
  x, & \textrm{if } i<j, \\[5pt] 
 (i-1)!,  & \textrm{if } i=j,  \\[5pt]
 0 , & \textrm{if } i>j.  
 \end{cases}
\end{equation}
Taking into account the sign of the entries  of $BD( {W})$ and Theorem \ref{thm1}, one can derive that  the Wronskian matrix of the monomial basis is nonsingular totally positive 
 for any  $x\geq0$. Moreover, the~computation of~\eqref{eq:BDWmonomials} satisfies the NIC condition and $W$ and its bidiagonal decomposition $BD(W)$ can be computed to high relative accuracy---in the case of $W$, this is achieved through the \verb"TNExpand" routine (see  \cite{Koev3}) which recovers a matrix from its bidiagonal decomposition, preserving high relative accuracy. Using~\eqref{eq:BDWmonomials}, computations to HRA  when solving algebraic problems related to $W$ have been achieved in \cite{MPRCalcolo} for $x\geq0$.

Moreover, Corollary 6 of  \cite{MPRNewton} provides the   factorization~\eqref{BDAfac} of 
 $W_{J}:=JWJ$, where $J$ is the diagonal matrix $J=\textrm{diag}((-1)^{i-1}  )_{1\le i\le n+1}$ and derives that $W_{J}$ is nonsingular totally positive for $x\leq0$.
 For the matrix representation $BD( {W_{J}})$, we have
\begin{equation}\label{eq:BDWJmonomials}
 BD( {W_{J}})_{i,j}:=\begin{cases}
  -x, & \textrm{if } i<j, \\[5pt] 
 (i-1)!,  & \textrm{if } i=j,  \\[5pt]
 0 , & \textrm{if } i>j.  
 \end{cases}
\end{equation}
Furthermore,  the~computation of~\eqref{eq:BDWJmonomials} satisfies the NIC condition and $W$ and $BD(W_{J})$ can be computed to HRA \cite{Koev3}. Using formula \eqref{eq:L},  it can be checked that  
\begin{equation}\label{eq:changeW}
W( {A}_0^{(q,\alpha)}, \ldots ,  {A}_n^{(q,\alpha)})(x)  = W(m_{0},\ldots,m_{n})(x) (L_n^{(q,\alpha)})^{T}
\end{equation}
and then we can write
\begin{equation}\label{eq:changeWJ}
JWJ = (JW(m_{0},\ldots,m_{n})(x)J) (J(L_n^{(q,\alpha)})^{T}J).
\end{equation}

So, with the reasoning of the proof of Theorem \ref{thm:collocation} and taking into account Theorem 2 presented in  \cite{MPRNewton}, the next result follows and we can also guarantee computations to HRA  when solving algebraic problems related to   Wronskian matrices of $q$-Abel polynomial bases.

\begin{theorem}  \label{thmWAbel}
 Let   $J:=\textrm{diag}((-1)^{i-1}  )_{1\le i\le n+1}$ and  $W:=W( {A}_0^{(q,\alpha)}, \ldots ,  {A}_n^{(q,\alpha)})(x)$   be  the Wronskian matrix of the q-Abel basis  defined in  \eqref{eqBaseabel}. Then
 
\begin{itemize} 
\item[a)] If $x\ge 0$ and $\alpha\leq0$ then    $W$ is  TP.  
\item[b)]  If   $x\le 0$ and $\alpha\geq0$ then  $W_{J} = JWJ$ is TP. 
\end{itemize}
Furthermore, in both cases, the matrices $W$, $W_{J}$ and their bidiagonal decompositions \eqref{eq:BDA}  can be computed to HRA.
Consequently, for $x\ge 0$ with $\alpha\leq0$ and  $x\le 0$ with $\alpha\geq0$, the eigenvalues, the singular values,  and  the inverse matrix of \,$W$, as well as the  solution  of linear systems $Wc = d$, where the entries of $d = (d_1, \ldots , d_{n+1})^T$ have alternating signs, for the case $x > 0$ and $\alpha\leq0$ and  the entries of $d = (d_1, \ldots , d_{n+1})^T$ have the same signs, for the case  $x< 0$ and $\alpha\geq0$, can be computed to HRA.
\end{theorem} 

 It is well known that the polynomial space  ${\mathbf P}^n([0,1])$    is a Hilbert space under the inner product
\begin{equation}\label{eq:innerp01}
\left<p,q\right> := \int_0^1 p(x)q(x)\, dx,
\end{equation}
and the  Gramian matrix of the monomial basis  $(m_0,\ldots,m_n)$  with respect to~\eqref{eq:innerp01}  is:
\begin{equation}\label{eq:Hilbert}
 H :=\left( \int_ 0^1 x^{i+j-2} \,dx   \right)_{1\le i,j\le n+1}=\left( \frac{1}{i+j-1}    \right)_{1\le i,j\le n+1}.
\end{equation}
The matrix $H_{n}$ is the $(n+1)\times (n+1)$ Hilbert matrix which is a particular case of a Cauchy matrix. In~Numerical Linear Algebra, Hilbert matrices are  well-known Hankel matrices. Their inverses and  determinants have explicit formulas; however, they are    very ill-conditioned  for moderate values of their dimension. Then,  they  can be used to test numerical algorithms and see how they perform on  ill-conditioned or nearly singular matrices.
  It is well known that Hilbert matrices are STP. In \cite{Koev4},  the~pivots and the multipliers of the   NE of $H$ are explicitly derived. It can be checked that  $BD(H)=(BD(H)_{i,j})_{1\le i,j\le n+1}$ is given by
\begin{equation}\label{eq:BDAMBernsteinN}
 BD(H)_{i,j}:=\begin{cases}
  \frac{ (i-1)^2}{(i+j-1)(i+j-2)}, & \textrm{if } i>j, \\[5pt]  
 \frac{  (i-1)!  ^4  }{(2i-1)!(2i-2)! },  & \textrm{if } i=j,   \\[5pt] 
  \frac{ (j-1)^2}{(i+j-1)(i+j-2)},  & \textrm{if } i<j. 
 \end{cases}
\end{equation}
Clearly, the~computation of the factorization~\eqref{BDAfac}  of $H$ does not require inaccurate cancellations and so, it  can be  computed to HRA.    

Using formula~\eqref{eq:L},  it can be checked that the $(n+1)\times (n+1)$ Gramian matrix   ${G}^{(q,\alpha)}$ with~respect to the inner product~\eqref{eq:innerp01},   of the $q$-Abel  basis  $(A_0^{(q,\alpha)},   \ldots, A_n^{(q,\alpha)})$ defined by \eqref{eqBaseabel}  ,  can be written as follows,
\begin{equation}\label{eq:GramqAbel}
{G}^{(q,\alpha)}=L^{(q,\alpha)}  {H} (L^{(q,\alpha)} )^T,
\end{equation}
 where  $L^{(q,\alpha)}$ is the $(n+1)\times (n+1)$ matrix given by \eqref{eq:qlij}. According to the reasoning  in the proof of Theorem \ref{thm:collocation},  the~following result can be~deduced.

\begin{theorem} \label{thm:GramqAbel}  If $\alpha{\le 0}$, the Gramian matrix  ${G}^{(q,\alpha)}$   is STP. Moreover, ${G}^{(q,\alpha)}$   and its bidiagonal decomposition ~\eqref{BDAfac} can be computed to HRA. Consequently, the eigenvalues,  singular values  and  the inverse matrix  of $G^{(q,\alpha)}$, as well as  the  solution  of linear systems $G^{(q,\alpha)}c = d$, where the entries of $d = (d_1, \ldots , d_{n+1})^T$ have alternating signs, can be computed to HRA.
\end{theorem}

Finally, it is worth noting that the results of this Section, valid for any positive $q$, hold in particular for $q=1$, which corresponds to the usual Abel polynomials. This implies, e.g., that, for $\alpha\le 0$ the basis of Abel polynomials is an STP basis of ${\mathbf P}^n(0,+\infty)$. In the next section, accurate computations are shown when solving algebraic problems with collocation, Wronskian and Gram matrices of the q-Abel basis for different values of $q$, including $q=1$.

\section{Numerical experiments}\label{sec:experimentos}
 
In order to illustrate the theoretical results obtained in previous sections, in what follows some numerical tests are presented. Three different matrices related to q-Abel bases, for $q>0$, have been considered:
\begin{itemize} 
\item Collocation matrices $A^{(q,\alpha)}$ \eqref{eq:change} for parameters   $0<t_1<\cdots<t_{n+1}$ with $\alpha\leq0$, and for parameters $0>t_1>\cdots>t_{n+1}$ with $\alpha\geq0$. 
\item  Wronskian matrices $W^{(q,\alpha)}$ \eqref{eq:changeW} for  $x\geq0$ with $\alpha\leq0$, and  for $x\leq0$ with $\alpha\geq0$.
\item Gramian matrices $G^{(q,\alpha)}$ \eqref{eq:GramqAbel} for  $\alpha\leq0$.
\end{itemize}

To give an idea of the difficulties that arise when addressing these matrices with traditional methods, its $2$-norm condition number---i.e., the ratio between the largest and the smallest singular values---was computed in Mathematica with 200-digit arithmetic and is presented in Figure \ref{fig:kappa2}. As can be seen, the conditioning rises severely with the size $n$ and, as a consequence, standard methods are vulnerable to inaccurate cancellations and ultimately fail to deliver accurate solutions.

\begin{figure}[h]
\begin{center}
\includegraphics[width=0.45\linewidth]{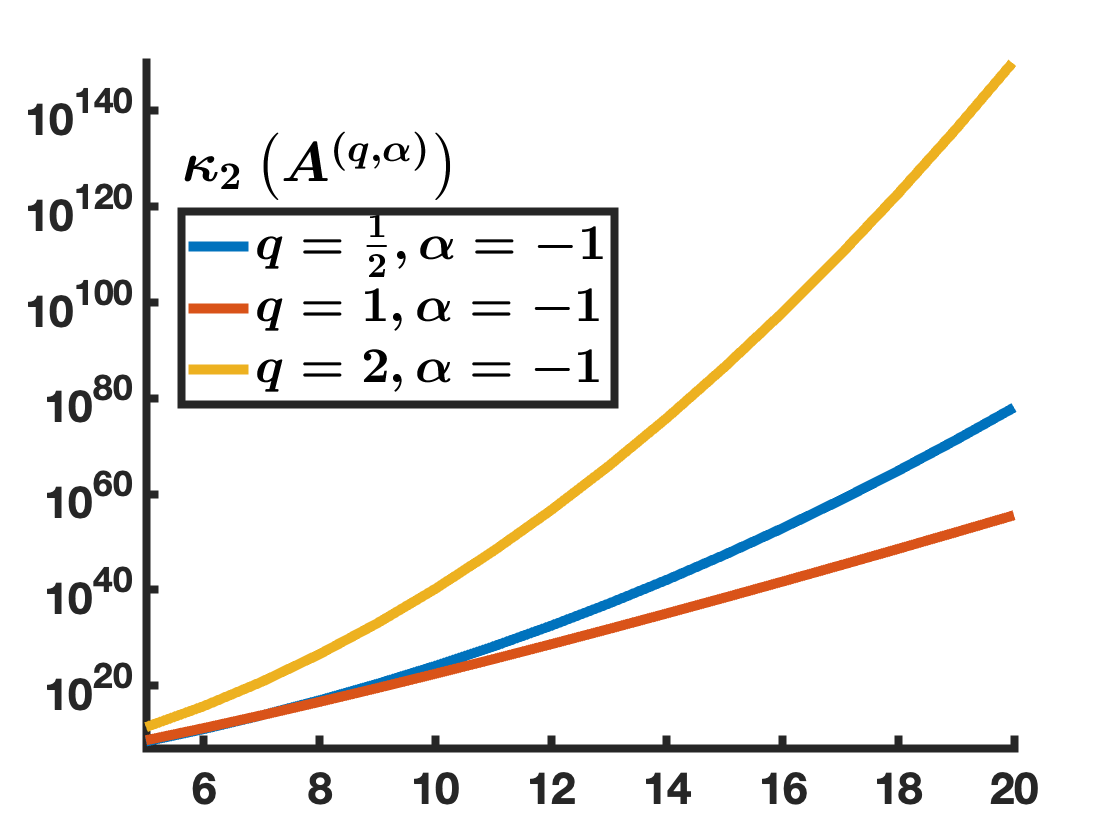}
\includegraphics[width=0.45\linewidth]{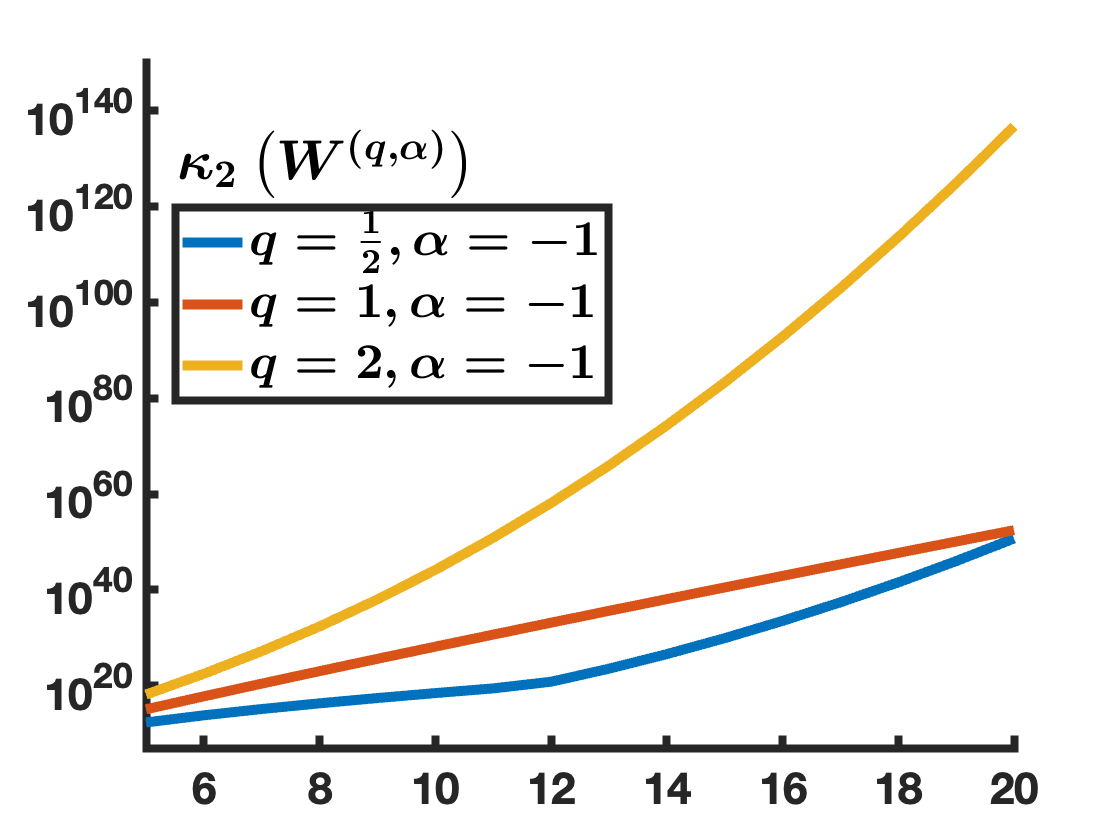}
\includegraphics[width=0.45\linewidth]{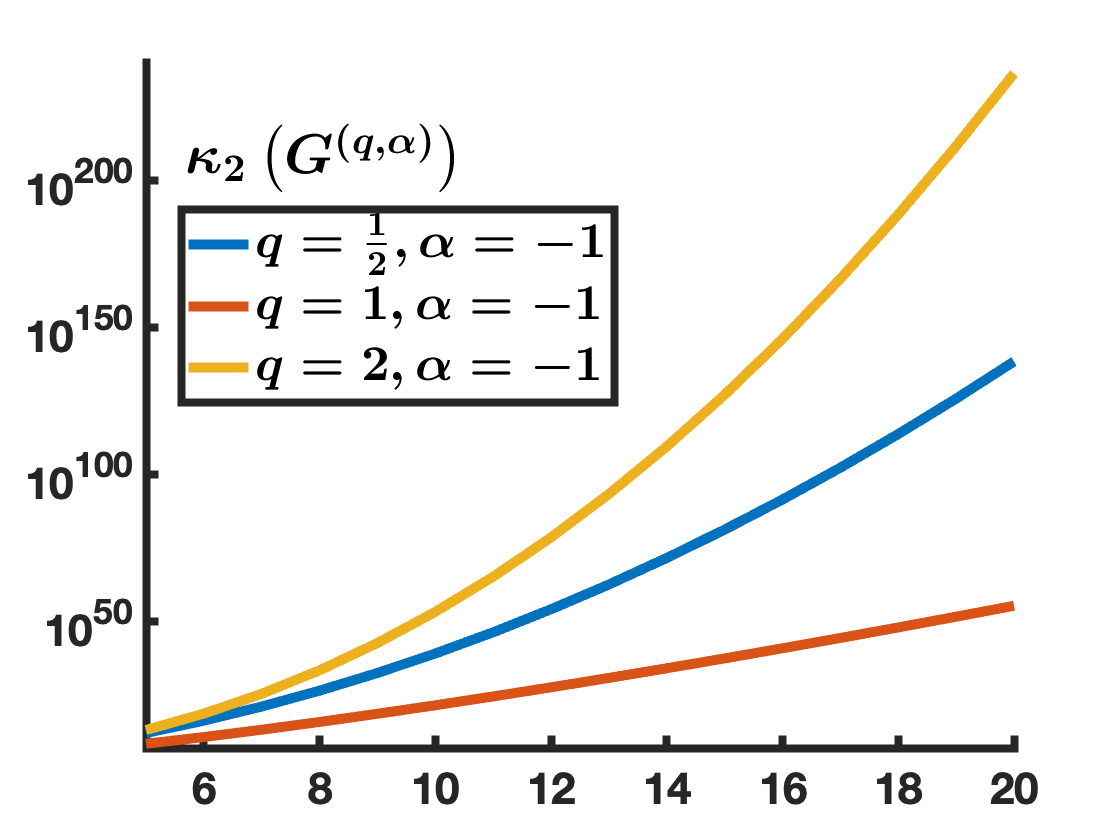}
\caption {\label{picture1} The 2-norm conditioning  of collocation matrices $A^{(q,\alpha)}$ \eqref{eq:change} with equidistant nodes $t_i=i/(n+1)$ for $i=1,\dots,n+1$, Wronskian matrices $W^{(q,\alpha)}$ \eqref{eq:changeW} for $x=50$ and Gramian matrices $G^{(q,\alpha)}$ \eqref{eq:GramqAbel} of $q$-Abel bases. \label{fig:kappa2}}
 \end{center}
\end{figure}

Nevertheless, for a nonsingular TP matrix $A$, as long as $BD(A)$ \eqref{eq:BDA} can be provided to high relative accuracy, the above mentioned algebraic problems can be solved preserving HRA. In order to do so, the functions \verb"TNSolve", \verb"TNEigenValues", \verb"TNSingularValues" and \verb"TNInverseExpand" (see  \cite{Marco3}), available in  Koev's software library \verb"TNTool" \cite{Koev3},  compute the solution of each algebraic problem, taking as input $BD(A)$. The function \verb"TNProduct"  is also  avaliable in the mentioned library. Let us recall that, given  $A=BD(F)$ and $B =BD(G)$ to high relative accuracy,  $\verb"TNProduct"(A,B)$ computes $ BD(FG)$ to high relative accuracy. The computational cost  of  \verb"TNSolve" and \verb"TNInverseExpand" is  $O(n^2)$, being $O(n^3)$ for the other functions.

For solving the different algebraic problems carried out in the numerical experimentation, it should be noted that the function \verb"TNProduct" has been used to compute the bidiagonal decompositions  \eqref{eq:BDA} necessary to achieve the solutions  to HRA. In Table~\ref{TableAlgorithms} the specific input arguments to obtain the different bidiagonal decompositions  are detailed, together with the necessary theoretical results.

\begin{table}[ht]\centering
\small

\begin{tabular}{@{}llll@{}}
\hline
Matrix & Parameters & Bidiagonal decomposition & Theoretical results  \\
\hline
$A^{(q,\alpha)}$  & $0<t_1<\cdots<t_{n+1}$ and $\alpha\leq0$  & $BD(A^{(q,\alpha)})=\verb"TNProd"(BD(V),BD(L^{(q,\alpha)})^T)$ & \eqref{eq:BDAbelq2}, \eqref{eq:BDV} and \eqref{eq:change} \\
$A^{(q,\alpha)}$ &  $0>t_1>\cdots>t_{n+1}$ and $\alpha\geq0$ & $BD(A_{J}^{(q,\alpha)})=\verb"TNProd"(BD(V_J),BD(L_J^{(q,\alpha)})^T)$ & \eqref{eq:BDAbelq2J}, \eqref{eq:BDVJ} and \eqref{eq:JWJ} \\

$W^{(q,\alpha)}$ &  $x\geq0$ and $\alpha\leq0$   & $BD(W^{(q,\alpha)})=\verb"TNProd"(BD(W),BD(L^{(q,\alpha)})^T)$ & \eqref{eq:BDAbelq2}, \eqref{eq:BDWmonomials} and \eqref{eq:changeW} \\

$W^{(q,\alpha)}$ &  $x\leq0$ and $\alpha\geq0$ & $BD(W_{J}^{(q,\alpha)})=\verb"TNProd"(BD(W_J), BD(L_J^{(q,\alpha)})^T)$ & \eqref{eq:BDAbelq2J}, \eqref{eq:BDWJmonomials} and \eqref{eq:changeWJ} \\
$G^{(q,\alpha)}$ & $\alpha<0$  & $BD(G^{(q,\alpha)})=$ & \eqref{eq:BDAbelq2}, \eqref{eq:BDAMBernsteinN} and \eqref{eq:GramqAbel} \\
& & $=\verb"TNProd"(\verb"TNProd"(BD(L^{(q,\alpha)}),BD(H)),BD(L^{(q,\alpha)})^T)$  & 
\end{tabular}

\caption{\label{TableAlgorithms} Algorithms applied to  obtain the proposed bidiagonal decompositions \eqref{eq:BDA} used in the numerical experimentation. }
\end{table}

In order to numerically try the goodness of the algorithms presented in this work, in each of the ensuing cases we have compared our proposal with the corresponding standard routine given in MATLAB/Octave. Additionally, the exact solution is taken to be the one provided by Wolfram Mathematica 13.3 with 200-digit arithmetic. Then, relative errors of each approximation are computed as $e:=||y-\tilde{y}||_2/||y||_2$, where $\tilde{y}$ is the approximated result and $y$ the one provided by Mathematica. Note that $||y||_2$ is the vector or matricial 2-norm, depending on the nature of $y$.

It is also worth noting that the experiments have been run at different matrix sizes $n$, to check its dependency, and with different values of the parameters $q$ and $\alpha$, including $q=1$, since in this case $q$-Abel polynomials are reduced to the classical Abel polynomials.

\textbf{Resolution of linear systems.} To begin with we address, first for negative $\alpha$, the collocation matrix $A^{(q,\alpha)}$ with $n+1$ equidistant nodes in increasing order $t_i=i/(n+1)$ for $i=1,\dots,n+1$, and the Wronskian matrix $W^{(q,\alpha)}$, for $x=50$. 
In second place, the positive $\alpha$ case is probed for the same systems but, for $A^{(q,\alpha)}$, with $n+1$ equidistant negative nodes in decreasing order $t_i=-i/(n+1)$, and for the Wronskian matrix $W^{(q,\alpha)}$, with $x=-20$. 

It is important to note that the system $W^{(q,\alpha)}y=b$, in the case of nonpositive $x$ and $\alpha\geq0$, is guaranteed to be solved to HRA as long as all the components of $b$ have the same sign; in the rest of the systems considered, $b$ with alternating signs is required to obtain HRA.

In all cases, the solutions of the systems $A^{(q,\alpha)}y=b$ and $W^{(q,\alpha)}y=b$ are determined both with the proposed bidiagonal decomposition (see Table \ref{TableAlgorithms}) as input to \verb"TNSolve" and with the standard $\setminus$ MATLAB command. 
Results are presented for $\alpha=-1$ in Table \ref{TableLinearSystems} and for $\alpha=1$ in Table \ref{TableLinearSystemsJ}, in both cases for $q=0.5,1,0.2$.
As can be seen in the relative errors depicted in Tables \ref{TableLinearSystems} and \ref{TableLinearSystemsJ}, for every value of $q$ and $n$ tested, the proposed method holds its accuracy. This is in contrast with the approximation obtained by the standard $\setminus$ command, which looses precision rapidly as $n$ increases.

\begin{table}[ht]\centering

\begin{tabular}{@{}cccccc@{}}
\hline
\multicolumn{2}{l}{} &  \multicolumn{2}{c}{\bf $A^{(q,\alpha)}y=b$ } & \multicolumn{2}{c}{\bf $W^{(q,\alpha)}y=b$ }  \\
\hline
$q$ & $n$  &  $ A^{(q,\alpha)} \setminus b$ & {\verb"TNSolve"$(BD(A^{(q,\alpha)}),b)$}  &   $W^{(q,\alpha)} \setminus b$ & {$\verb"TNSolve"(BD(W^{(q,\alpha)}),b)$}    \\
\hline
0.5 &  5 & $5.1e-09$ & $1.7e-16$ & $1.7e-15$ & $5.6e-17$\\
0.5 & 10 & $1.0e+00$ & $3.3e-16$ & $1.2e-10$ & $1.8e-16$\\
0.5 & 15 & $1.0e+00$ & $1.2e-15$ & $3.3e-7$ & $1.2e-15$\\
0.5 & 20 & $1.0e+00$ & $2.4e-16$ & $4.7e-02$ & $5.0e-16$\\
\hline
1 &  5 & $1.8e-11$ & $5.8e-16$ & $3.8e-16$ & $4.6e-17$\\
1 & 10 & $1.0e+00$ & $3.7e-16$ & $5.6e-13$ & $1.6e-16$\\
1 & 15 & $1.0e+00$ & $4.6e-16$ & $6.2e-10$ & $1.9e-16$\\
1 & 20 & $1.0e+00$ & $3.4e-16$ & $7.9e-07$ & $2.8e-16$\\
\hline
2 &  5 & $1.6e-10$ & $1.3e-16$ & $4.1e-15$ & $1.0e-16$\\
2 & 10 & $1.0e+00$ & $3.6e-16$ & $1.3e-11$ & $6.9e-17$\\
2 & 15 & $1.0e+00$ & $1.8e-16$ & $4.6e-02$ & $3.5e-16$\\
2 & 20 & $1.0e+00$ & $1.3e-15$ & $7.7e+15$ & $4.0e-16$\\
\hline
\end{tabular}

\caption{\label{TableLinearSystems} Relative errors of the approximations to the solution of the linear systems $A^{(q,\alpha)}y=b$ for $n+1$ equidistant nodes $t_i=i/(n+1)$ for $i=1,\dots,n+1$ and $W^{(q,\alpha)}y=b$ for $x=50$. The vector $b$ is composed by uniform-distributed random integers in $(0,1000]$, with alternating signs. In all cases, the parameter $\alpha=-1$.}
\end{table}

\begin{table}[ht]\centering

\begin{tabular}{@{}cccccc@{}}
\hline
\multicolumn{2}{l}{} &  \multicolumn{2}{c}{\bf $A^{(q,\alpha)}y=b$ } & \multicolumn{2}{c}{\bf $W^{(q,\alpha)}y=b$ }  \\
\hline
$q$ & $n$  &  $ A^{(q,\alpha)} \setminus b$ & {$J\cdot$\verb"TNSolve"$(BD(A_J^{(q,\alpha)}),b)$}  &   $W^{(q,\alpha)} \setminus b$ & {$J\cdot\verb"TNSolve"(BD(W_J^{(q,\alpha)}),Jb)$}    \\
\hline
0.5 &  5 & $5.1e-09$ & $1.4e-16$ & $2.5e-15$ & $1.6e-16$\\
0.5 & 10 & $1.0e+00$ & $2.3e-16$ & $8.4e-11$ & $1.8e-16$\\
0.5 & 15 & $1.0e+00$ & $1.1e-15$ & $1.7e-08$ & $1.2e-15$\\
0.5 & 20 & $1.0e+00$ & $5.8e-16$ & $2.0e-06$ & $3.4e-16$\\
\hline
1 &  5 & $1.9e-11$ & $4.1e-16$ & $9.8e-17$ & $2.0e-16$\\
1 & 10 & $1.0e+00$ & $2.3e-16$ & $4.1e-12$ & $2.6e-16$\\
1 & 15 & $1.0e+00$ & $5.3e-16$ & $1.3e-09$ & $2.2e-16$\\
1 & 20 & $1.0e+00$ & $4.6e-16$ & $5.6e-07$ & $4.8e-16$\\
\hline
2 &  5 & $1.6e-10$ & $8.3e-17$ & $2.2e-15$ & $1.5e-16$\\
2 & 10 & $1.0e+00$ & $2.6e-16$ & $2.3e-10$ & $1.7e-16$\\
2 & 15 & $1.0e+00$ & $1.3e-16$ & $2.8e+00$ & $2.2e-16$\\
2 & 20 & $1.0e+00$ & $1.9e-15$ & $4.8e+19$ & $4.2e-16$\\
\hline
\end{tabular}

\caption{\label{TableLinearSystemsJ} Relative errors of the approximations to the solution of the linear systems $A^{(q,\alpha)}y=b$ for $n+1$ equidistant nodes $t_i=-i/(n+1)$ for $i=1,\dots,n+1$ and $W^{(q,\alpha)}y=b$ for $x=-20$. The vector $b$ is composed by uniform-distributed random integers in $(0,1000]$, with alternating signs in the case $A^{(q,\alpha)}y=b$ and with equal signs for $W^{(q,\alpha)}y=b$. In all cases, the parameter $\alpha=1$.}
\end{table}

\textbf{Computation of eigenvalues and singular values.} In this case, since the Wronskian is triangular, i.e., its eigenvalues are exact, and the eigenvalues of the Gramian coincide with its singular values, the presented results only concern the collocation matrix $A^{(q,\alpha)}$. In this case the nodes are chosen to be $t_i=\text{log}(i+1)/\text{log}(n+2)$ for $i=1,\dots,n+1$, the parameter $\alpha=-10$ and again three values of $q$ are studied. As in the previous numerical experiment, the bidiagonal decomposition of $A^{(q,\alpha)}$ (see Table \ref{TableAlgorithms}) is used as an input to \verb"TNEigenValues" and \verb"TNSingularValues" routines, and compared with the standard MATLAB commands.

Relative errors of the smallest eigenvalues and singular values are shown in Table \ref{TableEigenvalues}. In this case it is clear that MATLAB commands \verb"eig" and \verb"svd" are not able to properly determine the smallest eigenvalue or singular value, even for small $n$, while the proposed bidiagonal decomposition approach succeeds in maintaining HRA at every matrix size tested.

\begin{table}[ht]\centering

\begin{tabular}{@{}cccccc@{}}
\hline
$q$ & $n$  &  \verb"eig"$(A^{(q,\alpha)})$ & {\verb"TNEigenValues"$(BD(A^{(q,\alpha)}))$}  &   \verb"svd"$(A^{(q,\alpha)})$ & {\verb"TNSingValues"$(BD(A^{(q,\alpha)}))$}    \\
\hline
0.5 &  5 & $3.9e-04$ & $4.1e-16$ & $7.1e-04$ & $2.7e-16$\\
0.5 & 10 & $3.3e+20$ & $9.1e-16$ & $3.5e+13$ & $1.9e-15$\\
0.5 & 15 & $1.2e+37$ & $2.7e-15$ & $5.8e+36$ & $2.8e-15$\\
0.5 & 20 & $4.2e+68$ & $1.0e-16$ & $6.4e+69$ & $1.8e-15$\\
\hline
1 &  5 & $2.1e-05$ & $3.8e-16$ & $8.0e-03$ & $1.0e-15$\\
1 & 10 & $2.1e+15$ & $3.9e-16$ & $2.8e+09$ & $7.3e-17$\\
1 & 15 & $5.8e+23$ & $2.3e-15$ & $4.2e+27$ & $2.7e-15$\\ 
1 & 20 & $2.0e+66$ & $4.6e-15$ & $1.5e+46$ & $7.0e-16$\\
\hline
2 &  5 & $6.8e-05$ & $1.1e-15$ & $3.4e+00$ & $6.8e-16$\\
2 & 10 & $3.6e+27$ & $1.4e-16$ & $6.1e+16$ & $1.7e-16$\\
2 & 15 & $1.3e+76$ & $2.5e-15$ & $1.0e+00$ & $8.8e-17$\\
2 & 20 & $4.7e+49$ & $3.6e-15$ & $1.0e+00$ & $5.0e-17$\\
\hline
\end{tabular}

\caption{\label{TableEigenvalues} Relative errors of the approximations to the smallest eigenvalue and singular value of $A^{(q,\alpha)}$ for $n+1$ nodes $t_i=\text{log}(i+1)/\text{log}(n+2)$ for $i=1,\dots,n+1$ and $\alpha=-10$.}
\end{table}

\textbf{Computation of inverses.} Finally, we compare the computation of the inverses of both the collocation matrix $A^{(q,\alpha)}$, for nodes $t_i=i^2/(n+1)^2$ with $i=1,\dots,n+1$, and of $G^{(q,\alpha)}$. The same three values for $q$ are analyzed, and $\alpha=-0.1$ is chosen in this case. The comparison between the relative errors of the standard MATLAB routine \verb"inv" and our bidiagonal decompositions (see Table \ref{TableAlgorithms}) together with \verb"TNInverseExpand" is shown in Table \ref{TableInverse}. Again, the theoretical results are heavily supported by the numerical evidence, since contrary to the standard MATLAB approach, the bidiagonal factorization procedure determines the inverses to HRA in all studied cases.

\begin{table}[ht]\centering

\begin{tabular}{@{}cccccc@{}}
\hline
$q$ & $n$  &  \verb"inv"$(A^{(q,\alpha)})$ & {\verb"TNInvExp"$(BD(A^{(q,\alpha)}))$}  &   \verb"inv"$(G^{(q,\alpha)})$ & {\verb"TNInvExp"$(BD(G^{(q,\alpha)}))$}    \\
\hline
0.5 &  5 & $1.3e-13$ & $1.3e-16$ & $1.3e-07$ & $1.3e-16$\\
0.5 & 10 & $1.5e-02$ & $8.4e-16$ & $1.0e+00$ & $6.0e-17$\\
0.5 & 15 & $1.0e+00$ & $1.0e-15$ & $1.0e+00$ & $1.1e-16$\\
0.5 & 20 & $1.0e+00$ & $2.6e-16$ & $1.0e+00$ & $2.4e-16$\\
\hline
1 &  5 & $1.1e-13$ & $1.7e-16$ & $1.5e-09$ & $1.4e-16$\\
1 & 10 & $3.2e-07$ & $5.6e-16$ & $1.0e+00$ & $3.4e-16$\\
1 & 15 & $9.0e-01$ & $8.2e-16$ & $1.0e+00$ & $7.6e-16$\\ 
1 & 20 & $1.0e+00$ & $1.3e-15$ & $1.0e+00$ & $3.0e-15$\\
\hline
2 &  5 & $1.0e-13$ & $1.3e-16$ & $9.8e-10$ & $2.0e-16$\\
2 & 10 & $4.0e-05$ & $3.3e-16$ & $1.0e+00$ & $1.0e-16$\\
2 & 15 & $1.0e+00$ & $7.5e-16$ & $1.0e+00$ & $5.9e-16$\\
2 & 20 & $1.0e+00$ & $7.1e-16$ & $1.0e+00$ & $6.2e-16$\\
\hline
\end{tabular}

\caption{\label{TableInverse} Relative errors of the approximations to the inverses of $A^{(q,\alpha)}$ for $n+1$ nodes $t_i=i^2/(n+1)^2$ for $i=1,\dots,n+1$ and of $G^{(q,\alpha)}$, for
$\alpha=-0.1$ in both cases.}
\end{table}

\section{Conclusions and final remarks} \label{conclusions}
Abel polynomials, as was mentioned at the beginning of this work, can be used 
to obtain any polynomial sequence of binomial type---a known result coming from 
umbral calculus. In this paper we have considered the $q$ generalization of this 
polynomials: the $q$-Abel polynomials. When addressing linear algebraic problems
involving collocation, Wronskian and Gram matrices of these polynomials, one may
face very ill-conditioned matrices, as has been shown in Section 
\ref{sec:experimentos}. As a consequence, traditional numerical methods are 
usually not capable of delivering an accurate solution.

Leaning on the theory of totally positive matrices, the bidiagonal factorization 
of the considered matrices has been obtained, allowing us to analyze in which cases 
the total positivity condition is fulfilled. Moreover, due to the found expressions 
of the multipliers and pivots of the Neville elimination process, we have been able 
to provide algorithms that solve several algebraic problems with high relative 
accuracy.

 \vspace{0.4 cm}
 
\noindent {\bf Acknowledgements} We are thankful for the helpful comments and suggestions of the anonymous referee, which have improved this paper.
 \vspace{0.4 cm}

\noindent {\bf Funding} This work was partially supported by Spanish research grants PID2022-138569NB-I00  (MCI/AEI)  and RED2022-134176-T (MCI/AEI) and by Gobierno de Arag\'{o}n (E41$\_$23R, S60$\_$23R).
 \vspace{0.4 cm}
 
\noindent {\bf Data Availability} The source code used to run the numerical experiments is available upon request.

\section*{Declarations}

\noindent {\bf Conflict of interest} This study does not have any conflicts to disclose.

\bibliography{bibliography}{}

\end{document}